\documentclass[12pt]{article}
\usepackage{amsmath,amsfonts,amsthm,mathrsfs,color}
\newtheorem{theorem}{Theorem}
\newtheorem{corollary}{Corollary}
\newtheorem{lemma}{Lemma}

\newtheorem{proposition}{Proposition}
\newtheorem{definition}{Definition}
\newtheorem{remark}{Remark}
\newtheorem{example}{Example}

\numberwithin{equation}{section}

\def\bbbr{{\rm I\!R}} 
\def\bbbn{{\rm I\!N}} 

\def\la{\lambda}
\def\be{\beta}

\def\al{\alpha}

\def\openE{{{\rm I}\kern-.16em {\rm E}}}
\def\NN{\bbbn}
\def\RR{\bbbr}

\def\eop{{\hfill\vbox{\hrule height .3pt
      \hbox{\vrule width.3pt height 7pt
      \kern 7pt
      \vrule width .3pt}
      \hrule height .3pt}} \par\bigskip}

\title{\large{\textsc{Learning Rates of Regression with $q$-norm Loss and Threshold $^\dag$\footnotetext{ }}}}

\author{
Ting Hu \\
{\small School of Mathematics and Statistics, Wuhan University } \\
{\small Luojia Hill, Wuhan 430072, China, tinghu@whu.edu.cn }\\
Yuan Yao\\
{\small School of
Mathematical Sciences, Peking University}\\
{\small Beijing 100871, China,
yuany@math.pku.edu.cn}}

\date{}

\begin{document}
\maketitle

\begin{abstract}
This paper studies some robust regression problems associated with the $q$-norm loss ($q\geq 1$) and the $\epsilon$-insensitive
$q$-norm loss in the reproducing kernel Hilbert space. We establish a variance-expectation bound under a priori noise
condition on the conditional distribution, which is the key technique to measure the error bound. Explicit learning rates will be given under
the approximation ability assumptions on the reproducing kernel Hilbert space.
\end{abstract}

\bigskip

\medskip


\baselineskip 18pt

\section{Introduction}
In this paper we consider regression with the $q$-norm loss $\psi_q$ with $q\geq1$ and an $\epsilon$-insensitive
$q$-norm loss $\psi_q^\epsilon$ (to be defined) with
a threshold $\epsilon>0$.
Here $\psi_q$ is the univariate function defined by $\psi_q(u)=|u|^q$. For a learning algorithm generated by a
regularization scheme in reproducing kernel Hilbert spaces,
learning rates and approximation error will be presented when $\epsilon$ is chosen
appropriately for balancing learning rates and sparsity.

For $q=1$, the regression problem is the classical statistical method of least absolute deviations which is more
robust than the least squares method
and is resistant to outliers in data \cite{Huber}. Its associated loss
 $\psi(u)=|u|,u\in \RR,$ is widely used in practical applications for robustness. In fact, for all $q<2,$
 the loss $\psi_q$ is less sensitive to outliers and is thus more robust than the square loss.
 Vapnik \cite{Vapnik} proposed an $\epsilon$-insensitive
 loss $\psi^\epsilon(u): \RR \to \RR_+$ to get sparsity in support vector regressions, which is defined by
 \begin{equation}\label{insen}
\psi^\epsilon(u)=
\left\{\begin{array}{ll} |u|-\epsilon ,  &\hbox{ if} \ |u|> \epsilon, \\
0, &\hbox{ if}\ |u|\leq \epsilon.
\end{array}\right.
\end{equation}
When fixing $\epsilon>0,$ error analysis was conducted in \cite{TCP}.
Xiang, Hu and Zhou \cite{XHZ,XHZ2} showed how to accelerate learning rates and preserve sparsity by adapting $\epsilon$.
In \cite{HXZ}, they discussed the convergence ability with flexible $\epsilon$ in an online algorithm.
For the quantile regression with $\epsilon=0$ and a pinball loss having different slopes in
different sides of the origin in $\RR$ \cite{KoenBass},
Steinwart and Christamann \cite{SteinChris,Steinwart} established comparison theorems and derived learning rates
under some noise conditions.

In this paper, we apply the $q$-norm loss $\psi_q$ with $q>1$ to improve the convexity of the insensitive loss $\psi$.
Our results show how the insensitive parameter $\epsilon$ that produces the sparsity can be chosen adaptively as the function of the sample size $\epsilon=\epsilon(T)\rightarrow0$ when $T\rightarrow\infty$, to affect the error rates of the learning algorithm (to be defined by (1.4)). Such results include some early studies as special cases.

In the sequel, assume that the input space $X$ is a compact metric space and the output space $Y=\RR$.
Let $\rho$ be a Borel probability measure on $Z:=X\times Y$,
$\rho_x(\cdot)$ be the conditional distribution of $\rho$ at each $x\in X$ and $\rho_X$
be the marginal distribution on $X$.
For a measurable function $f:X\rightarrow Y,$
the {\it generalization error} ${\cal E}(f)$ associated with the $q$-norm loss $\psi_q$,
is defined by
\begin{equation}\label{gener}
{\cal E}(f)=\int_Z\psi_q(y-f(x))d\rho.
\end{equation}
Denote $f_q:X\rightarrow Y$ as the minimizer of the generalization error ${\cal E}(f)$ over all
measurable functions. Its properties and the corresponding learning problem in the empirical risk minimization
framework were discussed in \cite{Yao}. When $q=1,$ the target function $f_q$ is a function containing
the medians of the conditional distribution for all $x\in X$.
For symmetric distributions,
the median is also the regression function, which is the conditional mean for given $X.$
 We aim at learning the minimizer $f_q$ from a sample ${\bf
z}=\{(x_i,y_i)\}_{i=1}^T \in Z^T,$ which
is assumed to be independently drawn
according to $\rho$.
Inspired by the $\epsilon$-insensitive
 loss \cite{Vapnik}, we introduce an $\epsilon$-insensitive $q$-norm loss $\psi^\epsilon_q$ which is defined by
 \begin{equation}\label{qnorm}
\psi^\epsilon_q(u)=
\left\{\begin{array}{ll}( |u|-\epsilon )^q,  &\hbox{ if} \ |u|> \epsilon, \\
0, &\hbox{ if}\ |u|\leq \epsilon.
\end{array}\right.
\end{equation}
Our learning task will be carried out by a regularization scheme in reproducing
kernel Hilbert spaces.
With a continuous, symmetric and positive semidefinite
function $K: X\times X \to \RR$ (called a Mercer kernel), the {\it
reproducing kernel Hilbert space} (RKHS) ${\cal H}_K$ is defined
as the completion of the span of $\{K_x=K(x,\cdot): x\in X\}$ with
the inner product $\langle\cdot, \cdot\rangle_K$ satisfying
$\langle K_x, K_u\rangle_K=K(x,u).$
 The regularization algorithm in the paper takes the form
 \begin{equation}\label{svm}
f_{{\bf z}}^\epsilon=\arg\min_{f\in {\cal
H}_K}\big\{\frac{1}{T}\sum_{t=1}^T\psi^\epsilon_q(f(x_t)-y_t)+\la\|f\|^2_K\big\}.
\end{equation}
Here $\la>0$ is a regularization parameter. Our learning rates are stated in terms of approximation or regularization error, noise conditions,
and the capacity of the RKHS. Our main goal is to study how the learned function $f_{{\bf z}}^\epsilon$ in (\ref{svm}) converges to the
target function $f_q.$
There is a large literature \cite{CWYZ,WYZ,SZone} in learning theory for studying {\it the approximation error} or {\it regularization error} ${\cal D}(\la)$
of the triple $(K,\rho,q)$ defined by
\begin{equation*}\label{approx}
{\cal D}(\la)=\min_{f\in {\cal H}_K}\big\{{\cal E}(f)-{\cal E}(f_q) +\la\|f\|_K^2\big\},\quad\la>0.
\end{equation*}
 {\it The regularization function} is defined as
 \begin{align}\label{regularfun}
f_\la=\arg\min_{f\in {\cal H}_K}\big\{{\cal E}(f)-{\cal E}(f_q) +\la\|f\|_K^2\big\}.
\end{align}
In the sequel, let $L_{\rho_X}^p$ with $p>0$  be the space of p integrable functions with respect to $\rho_X$
and $\|\cdot\|_{L_{\rho_X}^p}$ be  the norm in $L_{\rho_X}^p$.
A usual assumption on
the regularization error ${\cal D}(\la)$ which imposes certain smoothness on ${\cal H}_K$ is
\begin{equation}\label{app}
{\cal D}(\la)\leq {\cal D}_0\la^\be,\quad \forall \la>0
\end{equation}
with some $0<\be\leq1$ and ${\cal D}_0>0$.
\begin{remark}
Assumption (\ref{app}) always holds with $\beta=0$. When the target function $f_q\in{\cal H}_K$ and ${\cal H}_K$ is dense in $C(X)$
 which consists of bounded continuous functions on $X$, the approximation error ${\cal D}(\la)\rightarrow0$ as $\la\rightarrow0.$
 Thus, the decay (\ref{app}) is natural and can be illustrated in terms of interpolation spaces \cite{SZone}.
Define the integral operator $L_K: L^2_{\rho_X}\rightarrow L^2_{\rho_X}$ by $L_K(f)(x)=\int_XK(x,y)f(y)d\rho_X, x\in X,f\in L^2_{\rho_X}$ and suppose that
the minimizer
$f_q$ is in the range of $L_K^\nu$ with $0<\nu\leq\frac{1}{2}$.
When $q=1,$ the approximation error ${\cal D}(\la)$ can be $O(\la^{\frac{\nu}{1-\nu}})$ for quantile regression \cite{XHZ2}.
 When $q=2$,
${\cal D}(\la)=O(\la^{2\nu})$ for the least square. For other $q>1$, the associated loss $\psi_q$ is Lipschitz in a bounded domain and
the corresponding ${\cal D}(\la)$ can be characterized by the ${\cal K}$-functional \cite {CWYZ}, which can have the same polynomial decay as (\ref{app}).
\end{remark}

We assume that the conditional distribution $\rho_x(\cdot)$ is supported on $[-M,M],\ M>0$ at each $x$
and is non-degenerate, i.e. any non-empty open set of $Y$ has strictly positive measure, which ensures that the target function $f_q$ is unique.
Without loss of generality, let the support of $\rho_x(\cdot)$ be $[-\frac{1}{2},-\frac{1}{2}]$ at each $x\in X$ and our
analysis below is applicable for any $M>0.$ We will prove that in the next section.
It is natural to project values of the learned function $f_{{\bf z}}^\epsilon$ onto some interval by the
projection operator \cite{CWYZ,WYZ1}.
\begin{definition}
The projection operator $\pi$ on the space of measurable functions $f:X\rightarrow \RR$ onto the interval $[-1,1]$ is defined by
\begin{equation*}\label{pro}
\pi(f(x))=
\left\{\begin{array}{ll} 1 , & \hbox{if}\quad f(x) \geq 1, \\
f(x), & \hbox{if}\quad -1<f(x)<1,\\
-1 ,  &\hbox{if}\quad f(x) \leq -1.
\end{array}\right.
\end{equation*}
\end{definition}
To demonstrate our main result in the general case, we shall give the following learning rate in the special case when
$K$ is $C^\infty$.
\begin{theorem}\label{genrate}
Let $X\subset\RR^n$ and $K\in C^\infty(X\times X)$. Assume that $f_q\in {\cal H}_K$ with $q>1$, $\|f_q\|_\infty\leq \frac{1}{4}$ and
the
conditional distributions $\{\rho_x(\cdot)\}_{x\in X}$ have density
functions given by
\begin{equation}\label{densityexpre}
\frac{d \rho_x}{d y}(y) = \left\{\begin{array}{ll}
A|y- f_q (x)|^\varphi, & \hbox{if} \ |y- f_q
(x)| \leq \frac{1}{4}, \\ 0, & \hbox{otherwise,}
\end{array}\right.
\end{equation}
where $A=2^{2\varphi+1}(\varphi+1),\ \varphi>0$.
Take $\la=\epsilon=T^{-\frac{q+\varphi+1}{2(q+\varphi)}}$, then for any $0<\delta<1,$ with confidence
$1-\delta,$ we have
\begin{align*}
\|\pi(f^\epsilon_{\bf z})-f_q\|_{L^{q+\varphi+1}_{\rho_X}}\leq C'\sqrt{\log\frac{3}{\delta}} T^{-\frac{1}{2(q+\varphi)} },
\end{align*}
where $C'$ is a constant independent of $T$ or $\delta.$

\end{theorem}

To state our main result in the general case, we need a noise condition
on the measure $\rho$ introduced in \cite{Steinwart,SteinChris}.

\begin{definition}\label{noisecondition}
Let $0<p \leq \infty$ and $w>0$. We say that $\rho$ has a
$p$-average type $w$ if there exist two functions
$b$ and $a$ from $X$ to $\RR$ such that $\{ba^{w
}\}^{-1} \in L^p_{\rho_X}$ and for any $x\in X$ and $s\in
(0, a(x)]$, there holds
\begin{equation*}
 \rho_x(\{y: f_{q}(x)\leq y \leq f_{q}(x)+s\})\geq b (x) s^{w }
 \end{equation*}
 and
\begin{equation}\label{type}
 \rho_x(\{y:f_{q}(x)-s \leq y \leq f_{q}(x)\})\geq b(x)
 s^{w}.
 \end{equation}
\end{definition}
This assumption can be satisfied by many common conditional distributions
such as Guassian, students' t distributions and uniform distributions. In the following,
we will give an example to illustrate Definition 2 in detail. More examples can be found in \cite{Steinwart,SteinChris}.
\begin{example}
We assume that the conditional distributions $\{\rho_x(\cdot)\}_{x\in X}$ are Guassian distributions with a uniform variance $
\sigma>0$, i.e. $\frac{d\rho_x}{dy}(y)=\frac{1}{\sqrt{2\pi}\sigma}\exp\{-\frac{(y-u_x)^2}{2\sigma^2}\}$, where $\{u_x\}_{x
\in X}$
are expectations of the Gaussian distributions $\{\rho_x(\cdot)\}_{x\in X}$. It is not difficult to
check that the minimizer $f_\rho(x)$ can  take the value of $u_x$ at each $x\in X,$ then  for any $s\in (0,\sigma]$, there holds
\begin{align*}
 &\rho_x(\{y: f_{q}(x)\leq y \leq f_{q}(x)+s\}=\frac{1}{\sqrt{2\pi}\sigma}\int_{f_\rho(x)}^{f_\rho(x)+s}
 \exp\{-\frac{(y-u_x)^2}{2\sigma^2}\}dy\\
 &=\frac{1}{\sqrt{2\pi}\sigma}\int_{0}^{s}
 \exp\{-\frac{y^2}{2\sigma^2}\}dy\geq\frac{1}{\sqrt{2\pi}\sigma}\int_{0}^{s}
 \exp\{-\frac{s^2}{2\sigma^2}\}dy\geq\frac{e^{-\frac{1}{2}}}{\sqrt{2\pi}\sigma}s.
 \end{align*}
By similarity, we also have that $\rho_x(\{y: f_{q}(x)-s\leq y \leq f_{q}(x)\}\geq\frac{e^{-\frac{1}{2}}}{\sqrt{2\pi}\sigma}s$.
Thus, the measure $\rho$ has a
$\infty$-average type $1$.
\end{example}
Our error analysis is related to the capacity of the hypothesis space ${\cal H}_K$ which is measured by covering numbers.

\begin{definition}
For a subset $S$ of $C(X)$ and $\varepsilon>0$, the covering number ${\cal N}(S,\varepsilon)$ is the minimal integer $l\in \NN$ such that
there exist $l$ disks with radius $\varepsilon$ covering $S$.
\end{definition}
The covering numbers of balls $B_{R}=\{f\in{\cal H}_K:\|f\|_K\leq R\}$ with $R>0$ of the RKHS have been well understood
in the learning theory \cite{Zhou1,Zhou2}.
In this paper, we assume for some $k>0$ and $C_k>0$ that
\begin{equation}\label{cover}
\log{\cal N}(B_1,\varepsilon)\leq C_k\big(\frac{1}{\varepsilon}\big)^k,\ \forall \varepsilon>0.
\end{equation}
\begin{remark}
When $X$ is a bounded subset of $\RR^n$ and the RKHS ${\cal H}_K$ is a Sobolev space $H^m(X)$ with index $m$, it is shown
\cite{Zhou1} that the condition (\ref{cover}) holds true with $k=\frac{2n}{m}$.
If the kernel $K$ lies in the smooth space $C^\infty(X\times X),$ then (\ref{cover}) is satisfied for an arbitrarily small $k>0.$
Another common way to measure the capacity of ${\cal H}_K$ is the empirical covering number \cite{Zhang}, which is out of scope of our discussion in this paper.
\end{remark}

Denote
\begin{equation}\label{thetar}
\theta=\min\{\frac{2}{q+w},\frac{p}{p+1}\}\in (0,1],\qquad r=\frac{p(q+w)}{p+1}>0.
\end{equation}
The following learning rates in the general case will be proved in Section 4.
One need to point out that the proof of Theorem 2 is only applicable to
the case $q>1$. However, when $q=1$, it is a special case of quantile regression and the same learning rates as those of Theorem 2 can be found in \cite{XHZ,XHZ2}.
\begin{theorem}
Suppose that $\rho$  has a
$p$-average type $w$ for some $0<p\leq\infty$ and $\omega>0$. Assume that the regularization error condition (\ref{app}) is satisfied
for some $0<\be\leq1$ and (\ref{cover}) holds with $k>0$. Take $\la=T^{-\al}, \epsilon=T^{-\eta}$ with $0<\al\leq 1$, $0<\eta\leq \infty.$ Let $\xi>0.$
Then for any $0<\delta<1,$ with confidence $1-\delta,$ there holds
\begin{align}\label{generalrate}
\|\pi(f^\epsilon_{\bf z})-f_q\|_{L^{r}_{\rho_X}}\leq C^* \big(\log\frac{3}{\xi}\big)^2\sqrt{\log\frac{3}{\delta}}T^{-\Lambda}
\end{align}
where $C^*$ is a constant independent of $T$ or $\delta$ ,
\begin{equation*}\label{parameter}
\Lambda =\frac{1}{q+w}\min\big\{\eta,\al\be,1-\frac{q(1-\be)\al}{2},\frac{1}{2-\theta},\frac{1}{2+k-\theta}-\frac{k}{1+k}\vartheta \big\}
\end{equation*}
with
$\vartheta=\max\left\{\frac{\al-\eta}{2},\frac{\al(1-\be)}{2},\frac{\al}{2}+\frac{q(1-\be)\al}{4}-\frac{1}{2},\frac{\al}{2}-\frac{1}{2(2-\theta)}, \frac{[\al(2+k-\theta)-1](1+k)}{(2+k-\theta)(2+k)}+\xi\right\}\geq0$
provided that
\begin{equation}\label{retrict}
\vartheta<\frac{1+k}{k(2+k-\theta)}.
\end{equation}
\end{theorem}
\begin{corollary}
Let $X\subset\RR^n$, $K\in C^\infty(X\times X)$. Assume (\ref{app}) and (\ref{type}). Take $\la=T^{-1}$, $\epsilon=T^{-\eta}$ with $0<\eta\leq\infty$. If $1<q\leq2$, then
the index $\Lambda$ for the learning rate (\ref{generalrate}) is $\frac{1}{q+w}\min\{\eta,\be,\frac{1}{2-\theta}\}$.
\end{corollary}
\begin{remark} When $\eta=\infty$, the corresponding threshold $\epsilon$ is $0$ and it is a least square problem for $q=2$,
which is widely discussed in \cite{WYZ1,WYZ}. If  $\rho$ has a
$\infty$-average type $w$ with $w>0$ and $f_q\in {\cal H}_K$, the learning rate $\|\pi(f^\epsilon_{\bf z})-f_q\|_{L^{2+w}_{\rho_X}}=O(T^{-\frac{1}{2(1+w)}})$
for the least square.
It follows that
the error $\|\pi(f^\epsilon_{\bf z})-f_q\|^2_{L^{2}_{\rho_X}}=O(T^{-\frac{1}{1+w}})$ by $\|\cdot\|_{L^{2}_{\rho_X}}\leq\|\cdot\|
_{L^{2+w}_{\rho_X}}$. Thus, it can be near the optimal rate $O(T^{-1})$ in $L^{2}_{\rho_X}$ space if $w$ is small enough.

When $1<q<2,$ the learning error will be
 $O(T^{-\frac{1}{q+w}\min\{\be,\frac{q+w}{2(q+w-1)},\frac{p+1}{p+2}\}})$ with choice $\eta\geq\beta$, depending only on the
${\cal H}_K$'s approximation ability (\ref{app}) and noise condition (\ref{type}).
Specially, when $q$ goes to 1, it is the quantile regression \cite{XHZ,XHZ2} and the best rate is $O(T^{-\frac{1}{1+w}})$ in this paper
if $\rho$ has a
$\infty$-average type $w$ with $0<w\leq1$ and $f_q\in {\cal H}_K$.
\end{remark}

\section{Comparison and Perturbation Theorem}
Approximation or learning ability of a regularized algorithm for
regression problems can usually be studied by estimating the {\it
excess generalization error} ${\cal E}(f) - {\cal E}(f_{q})$
 for the learned function $f_{{\bf z}}^{\epsilon}$ from the algorithm (\ref{svm}).
However the following comparison theorem would yield bounds for the error $\|f-f_{q}\|_{L^{r}_{\rho_X}}$
in the space $L^{r}_{\rho_X}$ when the noise condition is satisfied.
\begin{theorem}
If $\rho$ has a p-average type $w$, then for any measurable function $f:X\rightarrow [-1, 1]$
we have the inequality
\begin{equation}\label{compa}
\|f-f_q\|_{L^{r}_{\rho_X}}\leq C_{r}\big({\cal E}(f) - {\cal E}(f_{q})\big)^{\frac{1}{q+w}}
\end{equation}
where the constant $C_{r}=2^{\frac{q-1}{q+w}}q^{-\frac{1}{q+w}}(q+w)^{\frac{1}{q+w}}\|(ba^w)^{-1}\|_{L^{p}_{\rho_X}}^{\frac{1}{q+w}}.$
\end{theorem}
\begin{proof}
For a measurable function $f:X\rightarrow [-1,1]$, the generalization error ${\cal E}(f)$ is rewritten as
${\cal E}(f)=\int_XC_{q,x}(f(x))d\rho_X$
where
\begin{equation*}
C_{q,x}(t)=\int_Y\psi_q(y-t)d\rho_x(y)=\int_{y>t}(y-t)^qd\rho_x(y)+\int_{y<t}(t-y)^qd\rho_x(y),\ x\in X.
\end{equation*}
Denote $t^*_x=\min_{t\in \RR}C_{q,x}(t).$ It is obvious that the minimizer $f_q(x)$ of ${\cal E}(f)$
takes the value of $t^*_x$ for each $x\in X$. Noting that the conditional distribution $\rho_x(\cdot)$ is supported
on $[-\frac{1}{2},\frac{1}{2}]$, the minimizer $t^*_x$ can be on $[-\frac{1}{2},\frac{1}{2}]$.  Consider the case $q>1.$
Since the loss function $\psi_q$ is differential and
 $|\frac{d\psi_q(y-t)}{dt}|\leq q|y-t|^{q-1}\leq q$ for all $ y, t\in [-\frac{1}{2},\frac{1}{2}]$, by the corollary of
 Lebesgue control convergence theorem, we can exchange the
order of  of integration and derivation of $C'_{q,x}(t)$ as
 $C'_{q,x}(t)=\frac{d}{dt}\int_Y\psi_q(y-t)d\rho_x(y)=\int_Y\frac{d\psi_q(y-t)}{dt}d\rho_x(y).$
This together with the fact $C'_{q,x}(t^*_x)=0,\forall x\in X,$ we have
\begin{align*}
C'_{q,x}(t^*_x)=q\int_{y<t^*_x}(t^*_x-y)^{q-1}d\rho_x(y)-
q\int_{y>t^*_x}(y-t^*_x)^{q-1}d\rho_x(y)=0,
\end{align*}
which means that
\begin{align}\label{Taylor}
\int_{y<t^*_x}(t^*_x-y)^{q-1}d\rho_x(y)=
\int_{y>t^*_x}(y-t^*_x)^{q-1}d\rho_x(y)
\end{align}
Let $t^*_x=0$ for simply, then we have
$C_{q,x}(t)-C_{q,x}(0)=\int_0^tC'_{q,x}(s)ds,\forall t>0$.
Noting that for $s>0$,
\begin{align*}
&C'_{q,x}(s)=q\big(\int_{y<s}(s-y)^{q-1}d\rho_x(y)-\int_{y>s}(y-s)^{q-1}d\rho_x(y)\big)\nonumber\\
&=q\Big(\int_{y<0}(s-y)^{q-1}d\rho_x(y)+\int_{0\leq y<s}(s-y)^{q-1}d\rho_x(y)-\int_{y>s}(y-s)^{q-1}d\rho_x(y)\Big)\nonumber\\
&\geq q\Big(\int_{y<0}(-y)^{q-1}d\rho_x(y)+\int_{0\leq y<s}(s-y)^{q-1}d\rho_x(y)-\int_{y>s}(y-s)^{q-1}d\rho_x(y)\Big).
\end{align*}
The above first term together with (\ref{Taylor}), then
\begin{align*}
&C'_{q,x}(s)\geq q\Big(\int_{y>0}y^{q-1}d\rho_x(y)+\int_{0\leq y<s}(s-y)^{q-1}d\rho_x(y)-\int_{y>s}(y-s)^{q-1}d\rho_x(y)\Big)\nonumber\\
&\geq q\Big(\int_{y>0}y^{q-1}d\rho_x(y)+\int_{0\leq y<s}(s-y)^{q-1}d\rho_x(y)-\int_{y>s}y^{q-1}d\rho_x(y)\Big)\nonumber\\
&= q\Big(\int_{0<y\leq s}y^{q-1}d\rho_x(y)+\int_{0\leq y<s}(s-y)^{q-1}d\rho_x(y)\Big)\nonumber\\
&= q\int_{0\leq y\leq s}\Big(y^{q-1}+(s-y)^{q-1}\Big)d\rho_x(y)\geq 2^{1-q}qs^{q-1}\rho_x(\{y:0\leq y\leq s\}).
\end{align*}
Thus, $$C_{q,x}(t)-C_{q,x}(0)\geq 2^{1-q}\cdot q\int_0^t s^{q-1}\rho_x(\{y:0\leq y\leq s\})ds.$$
Let us consider the first case $t\in[0,a(x)].$ Noting the noise condition (\ref{type}) and $a(x)\leq 1$, we obtain that
$$C_{q,x}(t)-C_{q,x}(0)\geq 2^{1-q}\cdot q\int_0^ts^{q-1}b(x)s^wds=\frac{2^{1-q}q}{q+w}b(x)t^{q+w}\geq \frac{2^{1-q}q}{q+w}b(x)a(x)^wt^{q+w}.$$
For the second case $t\in[a(x),1],$ we have
\begin{align*}
C_{q,x}(t)-C_{q,x}(0)&\geq2^{1-q}\cdot q\Big(\int_0^{a(x)} s^{q-1}\rho_x(\{y:0\leq y\leq s\})ds+\int^t_{a(x)} s^{q-1}\rho_x(\{y:0\leq y\leq s\})ds\Big)\\
&\geq 2^{1-q}\cdot q\Big(\int_0^{a(x)}s^{q-1}b(x)s^w ds+\int^t_{a(x)}s^{q-1}b(x)a(x)^wds\Big)\\
&=2^{1-q}\cdot q\Big(\frac{b(x)a(x)^wt^q}{q}-\frac{wb(x)a(x)^{q+w}}{q(q+w)}\Big)
\geq2^{1-q}\big(b(x)a(x)^wt^q-\frac{wb(x)a(x)^{w}t^q}{q+w}\big)\\
&=\frac{2^{1-q}q}{q+w}b(x)a(x)^wt^{q}\geq\frac{2^{1-q}q}{q+w}b(x)a(x)^wt^{q+w}.
\end{align*}
In general, we can see that for any $0<t\leq1$,
\begin{equation}\label{tm}
C_{q,x}(t)-C_{q,x}(0)\geq \frac{2^{1-q}q}{q+w}b(x)a(x)^wt^{q+w}.
\end{equation}
By similarity, if $-1\leq t<0$, we also have
\begin{equation}\label{tl}
C_{q,x}(t)-C_{q,x}(0)\geq \frac{2^{1-q}q}{q+w}b(x)a(x)^wt^{q+w}.
\end{equation}
Applying the two above inequalities (\ref{tm}) and (\ref{tl}) with $t=f(x)$ and $t^*_x=f_q(x),$ we have that
$$|f(x)-f_q(x)|^{q+w}\leq 2^{q-1}q^{-1}(q+w)(b(x)a(x)^w)^{-1}\big(C_{q,x}(f(x))-C_{q,x}(f_q(x))\big).$$
By $\frac{p}{p+1}$ power and integration,
\begin{align*}
&\int_X|f(x)-f_q(x)|^{\frac{p(q+w)}{p+1}}d\rho_X\leq2^{\frac{p(q-1)}{p+1}}q^{-\frac{p}{p+1}}(q+w)^{\frac{p}{p+1}}\\
&\int_X\big[(b(x)a(x)^w)^{-1}\big]^{\frac{p}{p+1}}\big(C_{q,x}(f(x))-C_{q,x}(f_q(x))\big)^{\frac{p}{p+1}}d\rho_X.
\end{align*}
This with Holder inequality $\|\cdot\|_{L^{1}_{\rho_X}}\leq \|\cdot\|_{L^{p^*}_{\rho_X}}\|\cdot\|_{L^{q^*}_{\rho_X}},
\frac{1}{p^*}+\frac{1}{q^*}=1$, we obtain that for $p^*=p+1$ and $q^*=\frac{p+1}{p},$
$$\int_X|f(x)-f_q(x)|^{\frac{p(q+w)}{p+1}}d\rho_X\leq2^{\frac{p(q-1)}{p+1}}q^{-\frac{p}{p+1}}(q+w)^{\frac{p}{p+1}}
\|(ba^w)^{-1}\|^{\frac{p}{p+1}}_{L^{p}_{\rho_X}}\big({\cal E}(f) - {\cal E}(f_{q})\big)^{\frac{p}{p+1}}.$$
Then the desired conclusion (\ref{compa}) holds. For $q=1,$ (\ref{compa}) also holds and the proof can be found in \cite{XHZ2}.
\end{proof}
It yields a variance-expectation bound which will be applied in the next section.
\begin{lemma}
Under the same conditions as Theorem 3, for any measurable function $f:X\rightarrow [-1, 1]$,
we have the inequality
\begin{equation}\label{variance}
\mathbb{E}\big\{(\psi_q(f(x)-y)-\psi_q(f_q(x)-y))^2\big\}\leq C_{\theta} ({\cal E}(f) - {\cal E}(f_{q})\big)^\theta
\end{equation}
where the power index $\theta$ is defined as (\ref{thetar}) and $C_{\theta}= C_{r}^2+2^{2-r}(1+\|f_q\|_\infty^{2-r})C_{r}^r$.
\end{lemma}
\begin{proof}
By the continuity of $\psi_q(u)$ and $|y|\leq\frac{1}{2}$, we see that
\begin{align}
&|\psi_q(f(x)-y)-\psi_q(f_q(x)-y)|\leq q(\|f\|_\infty^{q-1}+\|f_q\|_\infty^{q-1}+1)|f(x)-f_q(x)|\nonumber\\
&\leq q(2+\|f_q\|_\infty^{q-1})|f(x)-f_q(x)|\nonumber.
\end{align}
It implies that
$$\mathbb{E}\big\{(\psi_q(f(x)-y)-\psi_q(f_q(x)-y))^2\big\}\leq q^2(\|f_q\|_\infty^{q-1}+2)^2\mathbb{E}|f(x)-f_q(x)|^2.$$
If $r>2,$ then
\begin{equation*}
\mathbb{E}|f(x)-f_q(x)|^2\leq \big\{\mathbb{E}|f(x)-f_q(x)|^{r}\big\}^{\frac{2}{r}}\leq C_{r}^2({\cal E}(f) - {\cal E}(f_{q})\big)^{\frac{2}{q+w}}.
\end{equation*}
Else,
\begin{align}
&\mathbb{E}|f(x)-f_q(x)|^2\leq\mathbb{E}\big\{|f(x)-f_q(x)|^{2-r}\cdot|f(x)-f_q(x)|^{r}\big\}\nonumber\\
&\leq 2^{2-r}(\|f\|_\infty^{2-r}+\|f_q\|_\infty^{2-r})\mathbb{E}|f(x)-f_q(x)|^{r}\nonumber\\
&\leq 2^{2-r}(1+\|f_q\|_\infty^{2-r})C_{r}^r({\cal E}(f) - {\cal E}(f_{q})\big)^{\frac{p}{p+1}}\nonumber.
\end{align}
Combining the above two cases, we can get the conclusion (\ref{variance}).
\end{proof}

The threshold $\epsilon$ changes with the sample size $\epsilon=\epsilon(T)$ and plays a crucial role in
the design of algorithm (\ref{svm}). By Taylor expansion, we have the following relation
\begin{equation}\label{relation}
\psi_q^\epsilon(u)\leq \psi(u)\leq \psi_q^\epsilon(u)+q|u|^{q-1}\epsilon,\quad\forall u\in \RR.
\end{equation}
When the threshold $\epsilon\rightarrow0,$ the $\epsilon$-insensitive $q$-norm loss
$\psi_q^\epsilon$ converges to the $q$-norm function $\psi_q$ almost surely. In the following, we shall study the approximation of the
target function $f_q$ by $f_q^\epsilon$ which is the minimizer of the $\epsilon$-{\it generalization error}
${\cal E}^\epsilon(f)=\int_Z\psi^\epsilon_q(f(x)-y)d\rho$ for $\epsilon>0$. Denote
\begin{equation}\label{ecqx}
C_{q,x}^\epsilon(t)=\int_Y\psi_q^\epsilon(y-t)d\rho_x(y)=\int_{y>t+\epsilon}(y-t-\epsilon)^qd\rho_x(y)
+\int_{y<t-\epsilon}(t-y-\epsilon)^qd\rho_x(y),\ x\in X.
\end{equation}
and $t^\epsilon_x$ is the minimizer of $C_{q,x}^\epsilon(t)$.
By the same proof procedure as (\ref{Taylor}) in Theorem 3, we also get
\begin{align}\label{etaylor}
&\int_{y>t^\epsilon_x+\epsilon}(y-t^\epsilon_x-\epsilon)^{q-1}d\rho_x(y)=
\int_{y<t^\epsilon_x-\epsilon}(t^\epsilon_x-\epsilon-y)^{q-1}d\rho_x(y)\nonumber\\
\end{align}
and $f_q^\epsilon$ takes the value of $t^\epsilon_x$ at each $x\in X$.
 Then
the perturbation properties hold. We use some ideas from \cite{HFWZaa} in the proof.

\begin{proposition}
For $\epsilon>0,$ then
\begin{equation}\label{per1}
\|f_q^\epsilon-f_q\|_\infty\leq \epsilon.
\end{equation}
For any measurable function $f$ on $X,$ we have
\begin{equation}\label{per2}
{\cal E}(f) - {\cal E}(f_{q})\leq {\cal E}^\epsilon(f) - {\cal E}^\epsilon(f_{q}^\epsilon)+q(\|f\|_\infty^{q-1}+1)\epsilon
\end{equation}
\end{proposition}
\begin{proof}
Suppose that there exist a $x\in X$ satisfying $f_q^\epsilon(x)-f_q(x)> \epsilon$. Consider the case $q>1.$
Together with the fact (\ref{Taylor}) and $t_x^*=f_q(x)$, we note that
\begin{align}\label{equation1}
&\int_{y<f_q^\epsilon(x)-\epsilon}(f_q^\epsilon(x)-\epsilon-y)^{q-1}d\rho_x(y)\geq
\int_{y<f_q(x)}(f_q^\epsilon(x)-\epsilon-y)^{q-1}d\rho_x(y)\nonumber\\
&\geq\int_{y<f_q(x)}(f_q(x)-y)^{q-1}d\rho_x(y)=\int_{y>f_q(x)}(y-f_q(x))^{q-1}d\rho_x(y)
\end{align}
It is obvious that $f_q^\epsilon(x)+\epsilon>f_q(x)$ by the hypothesis that $f_q^\epsilon(x)-f_q(x)> \epsilon$ for any $\epsilon>0$.
By (\ref{etaylor}) with $t_x^\epsilon=f_q^\epsilon(x)$, we also get
\begin{align}\label{equation2}
&\int_{y>f_q(x)}(y-f_q(x))^{q-1}d\rho_x(y)\geq \int_{y>f_q^\epsilon(x)+\epsilon}(y-f_q(x))^{q-1}d\rho_x(y)\nonumber\\
&\geq\int_{y>f_q^\epsilon(x)+\epsilon}(y-f_q^\epsilon(x)-\epsilon)^{q-1}d\rho_x(y)=\int_{y<f_q^\epsilon(x)-\epsilon}(f_q^\epsilon(x)-\epsilon-y)^{q-1}d\rho_x(y).
\end{align}
Combining (\ref{equation2}) with (\ref{equation1}), we know that
\begin{align}
&\int_{y<f_q^\epsilon(x)-\epsilon}(f_q^\epsilon(x)-\epsilon-y)^{q-1}d\rho_x(y)=
\int_{y<f_q(x)}(f_q^\epsilon(x)-\epsilon-y)^{q-1}d\rho_x(y)\nonumber\\
&=\int_{y<f_q(x)}(f_q(x)-y)^{q-1}d\rho_x(y)=\int_{y>f_q(x)}(y-f_q(x))^{q-1}d\rho_x(y)\nonumber\\
&=\int_{y>f_q^\epsilon(x)+\epsilon}(y-f_q(x))^{q-1}d\rho_x(y)=\int_{y>f_q^\epsilon(x)+\epsilon}(y-f_q^\epsilon(x)-\epsilon)^{q-1}d\rho_x(y)
\end{align}

The above equalities hold if and only if $\rho_x(\{y:y>f_q(x)\})=0$ and $\rho_x(\{y:y<f_q^\epsilon(x)-\epsilon\})=0$ at the same time. Immediately,
we see that $\rho_x(\{y:y\leq f_q(x)\})=1-\rho_x(\{y:y>f_q(x)\})=1$.
By the hypothesis $f_q^\epsilon(x)-f_q(x)>\epsilon$, it follows that
\begin{align*}
\rho_x(\{y:y\leq f_q(x)\})\leq \rho_x(\{y:y<f_q^\epsilon(x)-\epsilon\})=0.
\end{align*}
This is contradiction. By similarity, we
get that $f_q^\epsilon(x)-f_q(x)< -\epsilon$ for each $x\in X$. Then the desired conclusion (\ref{per1}) holds.
By the relation (\ref{relation}) and $|y|\leq \frac{1}{2}$, we can see that
\begin{align*}
{\cal E}(f) - {\cal E}(f_{q})\leq {\cal E}^\epsilon(f) - {\cal E}^\epsilon(f_{q})+q\|f-y\|_\infty^{q-1}\epsilon\leq
{\cal E}^\epsilon(f) - {\cal E}^\epsilon(f_{q})+q(\|f\|_\infty^{q-1}+1)\epsilon.
\end{align*}
Then the desired conclusion (\ref{per2}) holds.
\end{proof}
We recall the fact that the conditional distribution $\rho_x(\cdot)$ is non-degenerate for each $x\in X,$
then the uniqueness of the minimizer $f_q^{\epsilon}$ is stated as following. For simply, we denote $f_q^\epsilon$ as the target function $f_q$
and ${\cal E}^\epsilon(f)$ as the generalization error ${\cal E}(f)$ with the $q$-norm loss $\psi_q$ when $\epsilon=0$ in the next proposition.
\begin{proposition}For $0\leq\epsilon\leq\frac{1}{2},$ the function $f_q^{\epsilon}$ is the unique minimizer of the
$\epsilon$-generalization error ${\cal E}^\epsilon(f)$.
\end{proposition}
\begin{proof}
Suppose that $f_q^{\epsilon}$ is not the unique minimizer.
For some $x\in X,$ there exists
$t_1(x)< t_2(x)$ such that they are both the minimizers of $C_{q,x}^\epsilon(t)$ by (\ref{ecqx})
and satisfy the equality (\ref{etaylor}) with $t^\epsilon_x=t_1(x)$ or $t^\epsilon_x=t_2(x)$ .
Applying (\ref{etaylor}) with $t^\epsilon_x=t_1(x)$ and $t_1(x)< t_2(x)$, it follows that
\begin{align*}
&\int_{y<t_2(x)-\epsilon}(t_2(x)-\epsilon-y)^{q-1}d\rho_x(y)\geq
\int_{y<t_1(x)-\epsilon}(t_2(x)-\epsilon-y)^{q-1}d\rho_x(y)\\
&\geq\int_{y<t_1(x)-\epsilon}(t_1(x)-\epsilon-y)^{q-1}d\rho_x(y)=\int_{y>t_1(x)+\epsilon}(y-t_1(x)-\epsilon)^{q-1}d\rho_x(y)\\
&\geq \int_{y>t_2(x)+\epsilon}(y-t_1(x)-\epsilon)^{q-1}d\rho_x(y)
\geq\int_{y>t_2(x)+\epsilon}(y-t_2(x)-\epsilon)^{q-1}d\rho_x(y).
\end{align*}
Applying (\ref{etaylor}) with $t^\epsilon_x=t_2(x)$ again, we see that the first term of the above inequality
$\int_{y<t_2(x)-\epsilon}(t_2(x)-\epsilon-y)^{q-1}d\rho_x(y)$ is equal to the last term $\int_{y>t_2(x)+\epsilon}(y-t_2(x)-\epsilon)^{q-1}d\rho_x(y)$.
This implies
\begin{align*}
&\int_{y<t_2(x)-\epsilon}(t_2(x)-\epsilon-y)^{q-1}d\rho_x(y)=
\int_{y<t_1(x)-\epsilon}(t_2(x)-\epsilon-y)^{q-1}d\rho_x(y)\\
&=\int_{y<t_1(x)-\epsilon}(t_1(x)-\epsilon-y)^{q-1}d\rho_x(y)=\int_{y>t_1(x)+\epsilon}(y-t_1(x)-\epsilon)^{q-1}d\rho_x(y)\\
&= \int_{y>t_2(x)+\epsilon}(y-t_1(x)-\epsilon)^{q-1}d\rho_x(y)
=\int_{y>t_2(x)+\epsilon}(y-t_2(x)-\epsilon)^{q-1}d\rho_x(y).
\end{align*}
The above equalities hold if and only if $\rho_x(\{y:y< t_2(x)-\epsilon\})=0$ and $\rho_x(\{y:y> t_1(x)+\epsilon\})=0$ simultaneously .
Since $\rho_x(\cdot)$ is non-degenerate and supported on $[-\frac{1}{2},\frac{1}{2}]$, then the values of $t_1(x)$ and $t_2(x)$ must satisfy
$ t_2(x)-\epsilon\leq -\frac{1}{2}$ and $ t_1(x)+\epsilon\geq\frac{1}{2}$. By the hypothesis $ t_1(x)<t_2(x)$,
we get $\epsilon>\frac{1}{2}$. This is
contradict with $0\leq\epsilon\leq\frac{1}{2}$. The proof is completed.
\end{proof}

\section{Error Decomposition and Sample Error}
Now we can conduct an error decomposition.
\begin{lemma}Define $f_\la$ by (\ref{regularfun}). Let $0\leq\epsilon\leq\frac{1}{2},$ then
\begin{equation}
{\cal E}(\pi(f_{\bf z}^\epsilon))-{\cal E}(f_q) +\la\|f_{\bf z}^\epsilon\|_K^2\leq S_1+S_2+{\cal D}(\la)+q2^{q-1}\epsilon,
\end{equation}
where \begin{align}
&S_1=\big[{\cal E}(\pi(f_{\bf z}^\epsilon))-{\cal E}(f_q)\big]-\big[{\cal E}_{\bf z}(\pi(f_{\bf z}^\epsilon))-{\cal E}_{\bf z}(f_q)\big],\\
&S_2=\big[{\cal E}_{\bf z}(f_\la)-{\cal E}_{\bf z}(f_q)\big]-\big[{\cal E}(f_\la)-{\cal E}(f_q)\big].
\end{align}
\end{lemma}
\begin{proof}
By the same procedure in \cite{SunWu,WuZhou,WYZ1,WYZ}, ${\cal E}(\pi(f_{\bf z}^\epsilon))-{\cal E}(f_q) +\la\|f_{\bf z}^\epsilon\|_K^2$ can be expressed as
\begin{align*}
&\big\{{\cal E}(\pi(f_{\bf z}^\epsilon))-{\cal E}_{\bf z}(\pi(f_{\bf z}^\epsilon))\big\}
+\big\{[{\cal E}_{\bf z}(\pi(f_{\bf z}^\epsilon))+\la\|f_{\bf z}^\epsilon\|_K^2]
-[{\cal E}_{\bf z}(f_\la)+\la\|f_{\la}\|_K^2] \big\}\\
&+\big\{{\cal E}_{\bf z}(f_\la)-{\cal E}(f_\la)\big\}+\big\{{\cal E}(f_\la)-{\cal E}(f_q) +\la\|f_\la\|_K^2\big\}.
\end{align*}
The relation (\ref{relation}) yields
\begin{align}\label{eq1}
&{\cal E}_{\bf z}(\pi(f_{\bf z}^\epsilon))=\frac{1}{T}\sum_{i=1}^T\psi_q(\pi(f_{\bf z}^\epsilon)(x_i)-y_i)\nonumber\\
&\leq \frac{1}{T}\sum_{i=1}^T\psi_q^\epsilon(\pi(f_{\bf z}^\epsilon)(x_i)-y_i)+
q(\|\pi(f_{\bf z}^\epsilon)\|_\infty+|y|)^{q-1}\epsilon\nonumber\\
&\leq {\cal E}_{\bf z}^\epsilon(\pi(f_{\bf z}^\epsilon))+q2^{q-1}\epsilon
\end{align}
and
\begin{align}\label{eq2}
{\cal E}^\epsilon_{\bf z}(f_\la)=\frac{1}{T}\sum_{i=1}^T\psi_q^\epsilon(f_\la(x_i)-y_i)
\leq\frac{1}{T}\sum_{i=1}^T\psi_q(f_\la(x_i)-y_i)={\cal E}_{\bf z}(f_\la).
\end{align}
The restriction $0\leq\epsilon\leq\frac{1}{2}$ implies
${\cal E}^\epsilon_{\bf z}(\pi(f_{\bf z}^\epsilon))\leq {\cal E}^\epsilon_{\bf z}(f_{\bf z}^\epsilon)$.
By (\ref{eq1}) and (\ref{eq2}),
 then we have
 \begin{align*}
 &[{\cal E}_{\bf z}(\pi(f_{\bf z}^\epsilon))+\la\|f_{\bf z}^\epsilon\|_K^2]
 -[{\cal E}_{\bf z}(f_\la)+\la\|f_{\la}\|_K^2]\\
 &\leq [{\cal E}_{\bf z}^\epsilon(\pi(f_{\bf z}^\epsilon))+\la\|f_{\bf z}^\epsilon\|_K^2]
 -[{\cal E}_{\bf z}(f_\la)+\la\|f_{\la}\|_K^2]+q2^{q-1}\epsilon\\
 &\leq [{\cal E}_{\bf z}^\epsilon(f_{\bf z}^\epsilon)+\la\|f_{\bf z}^\epsilon\|_K^2]
 -[{\cal E}_{\bf z}(f_\la)+\la\|f_{\la}\|_K^2]+q2^{q-1}\epsilon\\
 &\leq [{\cal E}_{\bf z}^\epsilon(f_{\bf z}^\epsilon)+\la\|f_{\bf z}^\epsilon\|_K^2]
 -[{\cal E}_{\bf z}^\epsilon(f_\la)+\la\|f_{\la}\|_K^2]+q2^{q-1}\epsilon.
 \end{align*}
Since $[{\cal E}_{\bf z}^\epsilon(f_{\bf z}^\epsilon)+\la\|f_{\bf z}^\epsilon\|_K^2]
 -[{\cal E}_{\bf z}^\epsilon(f_\la)+\la\|f_{\la}\|_K^2]\leq0,$ we have
  $$[{\cal E}_{\bf z}(\pi(f_{\bf z}^\epsilon))+\la\|f_{\bf z}^\epsilon\|_K^2]
-[{\cal E}_{\bf z}(f_\la)+\la\|f_{\la}\|_K^2]\leq q2^{q-1}\epsilon.$$
Then the desired conclusion holds.
\end{proof}
In the above error decomposition, the first two terms $S_1$ and $S_2$ are called {\it sample error}. For the second term $S_2,$
we get the following estimation.
\begin{corollary}Assume that (\ref{variance}),
there exists a subset $Z_{1,\delta}$ of $Z^T$ with measure at least $1-\frac{2\delta}{3}$
such that for any ${\bf z}\in Z_{1,\delta}$,
\end{corollary}
\begin{align}\label{s2bound}
S_2\leq \frac{7q(1+5\|f_\la\|^q_\infty)\log\frac{3}{\delta}}{6T}+\frac{2^{q+1}\log\frac{3}{\delta}}{3T}
+\big(\frac{2C_{\theta}\log\frac{3}{\delta}}{T}\big)^{\frac{1}{2-\theta}}+{\cal D}(\la).
\end{align}
\begin{proof}
we can decompose $S_2$ into two
parts $S_2=S_{2,1}+S_{2,2}$,
where
\begin{align*}
&S_{2,1}=\big[{\cal E}_{\bf z}(f_\la)-{\cal E}_{\bf z}(\pi(f_\la))\big]-\big[{\cal E}(f_\la)-{\cal E}(\pi(f_\la))\big],\\
&S_{2,2}=\big[{\cal E}_{\bf z}(\pi(f_\la))-{\cal E}_{\bf z}(f_q)\big]-\big[{\cal E}(\pi(f_\la))-{\cal E}(f_q)\big].
\end{align*}
For $S_{2,1},$ we apply the one-side Bernstein inequality \cite{DGL} to the random variable
$\xi(z)=\psi_q(f_\la(x)-y)-\psi_q(\pi(f_\la)(x)-y)$. For the continuity of the loss $\psi_q(u),$
it satisfies $0\leq \xi\leq q(\|f_\la\|_\infty+\|\pi(f_\la)\|_\infty+|y|)^{q-1}|\pi(f_\la)(x)-f_\la(x)|
\leq q(2\|f_\la\|_\infty^{q-1}+1)(1+\|f_\la\|_\infty)\leq q(1+5\|f_\la\|^q_\infty).$
Noting that $|\xi-\mathbb{E}(\xi)|\leq q(1+5\|f_\la\|^q_\infty)$ and $\mathbb{E}(\xi-\mathbb{E}(\xi))^2\leq q(1+5\|f_\la\|^q_\infty) \mathbb{E}(\xi),$
then there exists a subset $Z'_{1,\delta}$ of $Z^T$ with measure at least $1-\frac{\delta}{3}$
such that for any ${\bf z}\in Z'_{1,\delta}$,
\begin{equation}\label{s21bound}
S_{2,1}\leq \frac{7q(1+5\|f_\la\|^q_\infty)\log\frac{3}{\delta}}{6T}+{\cal E}(f_\la)-{\cal E}(\pi(f_\la)).
\end{equation}
For $S_{2,2},$ we take the random variable $\xi(z)=\psi_q(\pi(f_\la)(x)-y)-\psi_q(f_q(x)-y)$ which is bounded by $2^q$
and estimate the variance by Lemma 1 with $f=\pi(f_\la)$.  Applying the one-side Bernstein inequality again, we find
that there exists a subset ${\bf z}''_{1,\delta}$ of $Z^T$ with measure at least $1-\frac{\delta}{3}$
such that for any ${\bf z}\in Z''_{1,\delta}$,
\begin{equation}\label{s22bound}
S_{2,2}\leq \frac{2^{q+1}\log\frac{3}{\delta}}{3T}
+\big(\frac{2C_{\theta}\log\frac{3}{\delta}}{T}\big)^{\frac{1}{2-\theta}}+{\cal E}(\pi(f_\la))-{\cal E}(f_q)
\end{equation}
Combing the bound (\ref{s21bound}) and (\ref{s22bound}), we get the desired conclusion (\ref{s2bound}).
\end{proof}
Denote $\kappa=\sup_{x\in X}\sqrt{K(x,x)}.$ For $R\geq1$, let $B_R=\{{\bf z}\in Z^T:\|f\|_K\leq R\}$.
\begin{corollary}
Assume that (\ref{cover}) and (\ref{variance}). For any $f\in B_{R}$, there exists a subset $Z_{2,\delta}$ of $Z^T$ with measure at least $1-\frac{\delta}{3}$ such that for all ${\bf z}\in Z_{2,\delta},$
\begin{equation*}
S_1\leq \frac{1}{2}\big({\cal E}(\pi(f)-{\cal E}(f_q)\big)+12\varepsilon^*(R,T,\delta/3)
\end{equation*}
where
\begin{align}\label{s1bound}
&\varepsilon^*(R,T,\delta/3)\leq \left(2^{q+3}+(8C_{\theta})^{\frac{1}{2-\theta}}\right)\log\frac{3}{\delta}T^{-\frac{1}{2-\theta}}\nonumber\\
&\qquad+\left(2^{q+3}C_kq^k(2+\|f_q\|_\infty)^{k(q-1)}+(8C_{\theta}C_kq^k(2+\|f_q\|_\infty)^{k(q-1)})^{\frac{1}{2+k-\theta}}\right)R^{\frac{k}{k+1}}T^{-\frac{1}{2+k-\theta}}.
\end{align}
\end{corollary}
\begin{proof}
Consider the function set
$${\cal G}=\big\{\psi_q(\pi(f)(x)-y)-\psi_q(f_q(x)-y):\|f\|_K\leq R\big\}.$$
A function from this set $g(z)=\psi_q(\pi(f)(x)-y)-\psi_q(f_q(x)-y)$ satisfies
$\mathbb{E}g\geq0,$ $|g(z)|\leq 2^q$ and $\mathbb{E}g^2\leq C_{\theta}\big(\mathbb{E}g\big)^{\theta}$
by (\ref{variance}). The continuity of the loss implies
$|\psi_q(\pi(f)(x)-y)-\psi_q(f_q(x)-y)|\leq q(2+\|f_q\|_\infty)^{q-1}|\pi(f)(x)-f_q(x)|$. Then
$${\cal N}({\cal G},\varepsilon)\leq {\cal N}\left(B_1,\frac{\varepsilon}{q(2+\|f_q\|_\infty)^{q-1}R}\right).$$
We apply the ratio probability inequality with the covering number in \cite{WYZ},
\begin{align*}
&\mathrm{Prob}_{{\bf z}\in Z}\left\{\sup_{\|f\|_K\leq R}\frac{\big[{\cal E}(\pi(f)-{\cal E}(f_q)\big]-\big[{\cal E}_{\bf z}(\pi(f)-{\cal E}_{\bf z}(f_q)\big]}{\sqrt{\big({\cal E}(\pi(f)-{\cal E}(f_q)\big)^\theta+\varepsilon^\theta}}\leq 4\varepsilon^{1-\theta/2}\right\}\\
&\geq 1-{\cal N}\left(B,\frac{\varepsilon}{q(2+\|f_q\|_\infty)^{q-1}R}\right)\exp\left\{-\frac{T\varepsilon^{2-\theta}}{2C_{\theta}+2^{q+1}\varepsilon^{1-\theta}}\right\}\\
&\geq 1-\exp\left\{C_k\left(\frac{q(2+\|f_q\|_\infty)^{q-1}R}{\varepsilon}\right)^k-\frac{T\varepsilon^{2-\theta}}{2C_{\theta}+2^{q+1}\varepsilon^{1-\theta}}\right\}.
\end{align*}
We take $\varepsilon^*(R,T,\delta/3)$ to be the positive solution to the equation
\begin{equation*}
C_k\left(\frac{q(2+(\|f_q\|_\infty)R}{\varepsilon}\right)^k-\frac{T\varepsilon^{2-\theta}}{2C_{\theta}+2^{q+1}\varepsilon^{1-\theta/2}}=\log\frac{\delta}{3}.
\end{equation*}
It can be expressed as
\begin{align*}
&\varepsilon^{2+k-\theta}-\frac{2^{q+1}}{T}\log\frac{3}{\delta}\varepsilon^{1+k-\theta}-\frac{2C_{\theta}}{T}\log\frac{3}{\delta}\varepsilon^k
-\frac{2^{q+1}C_kq^k(2+\|f_q\|_\infty)^{k(q-1)}R^k}{T}\varepsilon^{1-\theta} \\&-\frac{2C_{\theta} C_kq^k(2+\|f_q\|_\infty)^{k(q-1)}R^k}{T} =0.
\end{align*}
The positive solution  $\varepsilon^*(R,T,\delta/3)$ to this equation can be bounded as
\begin{align}\label{ubound}
&\varepsilon^*(R,T,\delta/3)
\leq \max\big\{\frac{2^{q+3}}{T}\log\frac{3}{\delta},\left(\frac{8C_{\theta}}{T}\log\frac{3}{\delta}\right)^{\frac{1}{2-\theta}},\nonumber\\
&\left(\frac{2^{q+3}C_kq^k(2+\|f_q\|_\infty)^{k(q-1)}R^k}{T}\right)^{\frac{1}{1+k}},\left(\frac{8C_{\theta}C_kq^k(2+\|f_q\|_\infty)^{k(q-1)}R^k}{T}\right)^{\frac{1}{2+k-\theta}}\big\}.
\end{align}
Then there exists a subset $Z_{2,\delta}$ of $Z^T$ with measure at least $1-\frac{\delta}{3}$ such that for all ${\bf z}\in Z_{2,\delta},$
\begin{equation*}
\sup_{\|f\|_K\leq R}\frac{\big[{\cal E}(\pi(f)-{\cal E}(f_q)\big]-\big[{\cal E}_{\bf z}(\pi(f)-{\cal E}_{\bf z}(f_q)\big]}{\sqrt{\big({\cal E}(\pi(f)-{\cal E}(f_q)\big)^\theta+(\varepsilon^*(R,T,\delta/3))^\theta}}\leq 4(\varepsilon^*(R,T,\delta/3))^{1-\theta/2}.
\end{equation*}
For any ${\bf z}\in B(R)\bigcap Z_{2,\delta},$ we have
\begin{align*}
S_1&\leq 4(\varepsilon^*(R,T,\delta/3))^{1-\theta/2}\sqrt{\big({\cal E}(\pi(f)-{\cal E}(f_q)\big)^\theta+(\varepsilon^*(R,T,\delta/3))^\theta}\\
&\leq \frac{\theta}{2}\big({\cal E}(\pi(f)-{\cal E}(f_q)\big)+(1-\frac{\theta}{2})4^{1/(1-\theta/2)}\varepsilon^*(R,T,\delta/3)
+4\varepsilon^*(R,T,\delta/3)\\&\leq \frac{1}{2}\big({\cal E}(\pi(f)-{\cal E}(f_q)\big)+12\varepsilon^*(R,T,\delta/3).
\end{align*}
Putting the above bounds into (\ref{ubound}), then we get the desired conclusion (\ref{s1bound}).
\end{proof}

\section{Estimating Total Error by Iteration}\label{mainproof}
This section is devoted to estimating total error $\|\pi(f^\epsilon_{\bf z})-f_q\|_{L^r_{\rho_X}}.$ To apply Corollary 2 and Corollary 3 for error
analysis, we get the rough bound $$\|f_{\bf z}^\epsilon\|_K\leq \la^{-\frac{1}{2}},\quad \forall {\bf z}\in Z^T$$
by taking $f=0$ in (\ref{svm}). This bound will be improved by iteration technique used in \cite{WuZhou}.
For $R>0,$ denote
$${\cal W}(R)=\{{\bf z} \in Z^T:\|f^\epsilon_{\bf z}\|_K\leq R\}.$$
\begin{lemma}
Take $\la=T^{-\al}, \epsilon=T^{-\eta}$ with $0<\al\leq 1$, $0<\eta\leq \infty.$ Let $0<\xi<1.$
If $\rho$ satisfy the noise condition (\ref{type}) and (\ref{app}), (\ref{cover}) hold,
then for any $0<\delta<1,$ with confidence $1-\delta,$ there exists a subset $V_R$ of $Z^T$ with measure at most $\delta$ such that
holds
\begin{equation}\label{fzbound}
\|f_{\bf z}^\epsilon\|_K\leq 4A_2(1+\sqrt{q2^{q}}+2\sqrt{{\cal D}_0}
 +\sqrt{12q{\cal D}_0^{q/2}}+\sqrt{A_1})\big(\log\frac{3}{\xi}\big)^2\sqrt{\log\frac{3}{\delta}}T^\vartheta,\
 \forall {\bf z}\in {\cal W}(R)\setminus V(R),
\end{equation}
where $\vartheta=\max\big\{\frac{[\al(2+k-\theta)-1](1+k)}{(2+k-\theta)(2+k)}+\xi,\frac{\al-\eta}{2},\frac{\al(1-\be)}{2},
\frac{\al}{2}+\frac{q(1-\be)\al}{4}-\frac{1}{2},\frac{\al}{2}-\frac{1}{2(2-\theta)}\big\}$.
\begin{proof}
Applying Corollary 2 and Corollary 3 with Lemma 2, we know that for any
${\bf z}\in {\cal W}(R)\bigcap Z_{1,\delta}\bigcap Z_{2,\delta}, R>1,$
\begin{align*}
 {\cal E}(\pi(f_{\bf z}^\epsilon))-{\cal E}(f_q) +\la\|f_{\bf z}^\epsilon\|_K^2&\leq q2^{q}\epsilon
+4{\cal D}(\la) +12q\|f_\la\|_\infty^qT^{-1}\log\frac{3}{\delta}\\
 &+A_1T^{-\frac{1}{2-\theta}}\log\frac{3}{\delta}
 +A_2R^{\frac{k}{k+1}}T^{-\frac{1}{2+k-\theta}}.
 \end{align*}
 where $A_1$ and $A_2$ is given by
 \begin{align*}
A_1=106+20C_{\theta}^{\frac{1}{2-\theta}},
 A_2=2^{q+4}(C_kq^k(2+\|f_q\|_\infty)^{k(q-1)})^{\frac{1}{1+k}}+16(C_{\theta}C_kq^k(2+\|f_q\|_\infty)^{k(q-1)})^{\frac{1}{2+k-\theta}}.\\
 \end{align*}
 Let $V_R$ be a set whose measure is at most $\delta.$
 Putting $\la=T^{-\al}, \epsilon=T^{-\eta}$ with $0<\al\leq 1$, $0<\eta\leq \infty$  and (\ref{app}) into
 the above bound, then for any $R>1$ we have
 \begin{equation*}
 \|f_{\bf z}^\epsilon\|_K\leq a_TR^{\frac{k}{2+2k}}+b_T,\quad {\bf z}\in {\cal W}(R)\setminus V_R,
 \end{equation*}
 where the constants $a_T$ and $b_T$ are given by
 \begin{equation*}
 a_T=\sqrt{A_2}T^{\frac{\al}{2}-\frac{1}{2(2+k-\theta)}}, b_T=\left\{\sqrt{q2^{q}}+2\sqrt{{\cal D}_0}
 +\sqrt{12q{\cal D}_0^{q/2}\log\frac{3}{\delta}}+\sqrt{A_1\log\frac{3}{\delta}}\right\}T^{\zeta}
 \end{equation*}
 with $\zeta=\max\left\{\frac{\al-\eta}{2},\frac{\al(1-\be)}{2},\frac{\al}{2}+\frac{q(1-\be)\al}{4}-\frac{1}{2},\frac{\al}{2}-\frac{1}{2(2-\theta)} \right\}.$
 It follows that
 \begin{equation*}
 {\cal W}(R)\subseteq {\cal W}( a_TR^{\frac{k}{2+2k}}+b_T)\cup  V_R,
 \end{equation*}
 Let us apply the above relation iteratively to a sequence $\{R^{(j)}\}_{j=0}^J$ defined by $R^{(0)}=\la^{-\frac{1}{2}}$ and
 $R^{(j)}=a_T\big(R^{(j-1)}\big)^{\frac{k}{2+2k}}+b_T$ where $J\in\NN$ will be determined later. Then
 ${\cal W}(R^{(j-1)})\subseteq {\cal W}(R^{(j)})\cup V_{R^{(j-1)}}.$ Noting that ${\cal W}(R^{(0)})=Z^T,$ then
 $$Z^T={\cal W}(R^{(0)})\subseteq{\cal W}(R^{(1)})\cup V_{R^{(0)}}\subseteq\cdots{\cal W}(R^{(J)})
 \cup\big(\cup_{j=0}^{J-1}V_{R^{(j)}}\big).$$
 As the measure of $V_{R^{(j)}}$ is at most $\delta,$ we know that the measure of $\cup_{j=0}^{J-1}V_{R^{(j)}}$ is at most $J\delta.$
 Hence ${\cal W}(R^{(J)})$ has measure at least $1-J\delta.$

 Denote $\Delta=\frac{k}{2+2k}\leq \frac{1}{2}$. The definition of the sequence $\{R^{(j)}\}_{j=0}^J$ implies that
 $$R^{(J)}=a_T^{1+\Delta+\Delta^2+\cdots+\Delta^{J-1}}\big(R^{(0)}\big)^{\Delta^J}+\sum_{j=1}^{J-1}
 a_T^{1+\Delta+\Delta^2+\cdots+\Delta^{j-1}}b_m^{\Delta^j}+b_m.$$
 The first term
 \begin{align*}
 a_T^{1+\Delta+\Delta^2+\cdots+\Delta^{J-1}}\big(R^{(0)}\big)^{\Delta^J}&=
(A_2)^{\frac{1-\Delta^J}{2(1-\Delta)}}T^{\big(\frac{\al}{2}-\frac{1}{2(2+k-\theta)}\big)
 \frac{1-\Delta^J}{1-\Delta}}T^{\frac{\al}{2}\Delta^J}\\
 &\leq A_2 T^{\frac{[\al(2+k-\theta)-1](1+k)}{(2+k-\theta)(2+k)}}T^{\frac{1}{2+k-\theta}2^{-J}}.
 \end{align*}
Taking $J$ be the smallest integer greater than or equal to $\log\frac{1}{\xi}/\log2$. Then the upper bound is estimated by
$A_2 T^{\frac{[\al(2+k-\theta)-1](1+k)}{(2+k-\theta)(2+k)}+\xi}$.
The second term
\begin{align*}
\sum_{j=1}^{J-1}
 &a_T^{1+\Delta+\Delta^2+\cdots+\Delta^{j-1}}b_m^{\Delta^j}+b_m\leq A_2
 T^{\big(\frac{\al}{2}-\frac{1}{2(2+k-\theta)}\big)\frac{1-\Delta^j}{1-\Delta}}b_1^{\Delta^j}m^{\zeta\Delta^j}+b_1m^\zeta\\
 &\leq A_2b_1T^{\frac{[\al(2+k-\theta)-1](1+k)}{(2+k-\theta)(2+k)}}
 \sum_{j=0}^{J-1}T^{\big(\zeta-\frac{[\al(2+k-\theta)-1](1+k)}{(2+k-\theta)(2+k)}\big)\frac{k^j}{(2+2k)^j}}.
\end{align*}
where $b_1=\sqrt{q2^{q}}+2\sqrt{{\cal D}_0}
 +\sqrt{12q{\cal D}_0^{q/2}\log\frac{3}{\delta}}+\sqrt{A_1\log\frac{3}{\delta}}.$\\
 If $\zeta>\frac{[\al(2+k-\theta)-1](1+k)}{(2+k-\theta)(2+k)},$ it is bounded by $A_2b_1JT^\zeta.$ If $\zeta\leq\frac{[\al(2+k-\theta)-1](1+k)}{(2+k-\theta)(2+k)},$ it is bounded by $A_2b_1JT^{\frac{[\al(2+k-\theta)-1](1+k)}{(2+k-\theta)(2+k)}}.$\\
 Thus we have
 $$R^{(J)}\leq (A_2+A_2b_1J)T^\vartheta,$$
 where $\vartheta=\max\{\frac{[\al(2+k-\theta)-1](1+k)}{(2+k-\theta)(2+k)}+\xi,\zeta\}$.
 With confidence $1-J\delta,$ there holds
 \begin{equation*}
 \|f_{\bf z}^\epsilon\|_K\leq R^{(J)}\leq A_2(1+\sqrt{q2^{q}}+2\sqrt{{\cal D}_0}
 +\sqrt{12q{\cal D}_0^{q/2}}+\sqrt{A_1})\sqrt{\log\frac{3}{\delta}}JT^\vartheta.
 \end{equation*}
Noting $J\leq 2\log\frac{3}{\xi},$ then we can get (\ref{fzbound}) by replacing $\delta$ by $\delta/J$.
 \end{proof}
\end{lemma}

Now we can prove Theorem 2.\\
{\it Proof of Theorem 2.} By Lemma 3, there exists a subset $V_{R'}\subset Z^T$ with measure at most $\delta$ such
that $Z^T\setminus V_{R'}\subseteq {\cal W}(R).$ Let $R$ be the right side of (\ref{fzbound}). Applying Corollary 2 and
 Corollary 3 to $R$, then there exists another subset $V_R\subset Z^T$ with measure at most $\delta$ such that
 \begin{align*}
  {\cal E}(\pi(f))-{\cal E}(f_q)&\leq q2^{q}\epsilon
+4{\cal D}(\la) +12q\|f_\la\|_\infty^qT^{-1}\log\frac{3}{\delta}+A_1T^{-\frac{1}{2-\theta}}\log\frac{3}{\delta}\\
 &+A_3\big(\log\frac{3}{\xi}\big)^2\sqrt{\log\frac{3}{\delta}}T^{\frac{k}{1+k}\vartheta-\frac{1}{2+k-\theta}}.
 \end{align*}
where $A_3=A_2(4A_2)^{\frac{k}{k+1}}(1+\sqrt{q2^{q}}+2\sqrt{{\cal D}_0}
 +\sqrt{12q{\cal D}_0^{q/2}}+\sqrt{A_1})$.
By (\ref{compa}), we obtain that
\begin{align*}
\|\pi(f_{\bf z}^\epsilon)-f_q\|_{L^{r}_{\rho_X}}\leq C^*T^{-\Lambda}
\end{align*}
where
$$C^*=C_{r}\big(q2^{q}+4{\cal D}_0+12q{\cal D}_0^{q/2}+A_1+ A_3\big)^{\frac{1}{q+w}}$$
and $\Lambda$ is given by (\ref{parameter}).
The restriction (\ref{retrict}) ensures that $\Lambda>0.$ Replacing $\delta$ with $\delta/2,$ we
complete the proof of Theorem 2.\\
Now we are in the state of proving Theorem 1.\\
{\it Proof of Theorem 1.} We shall prove Theorem 1 by Theorem 2.
First, we check the noise condition (\ref{type}). Let the function $a(x)=\frac{1}{4}$ and $b(x)=2^{2\varphi+1},$
 $\forall x\in X$. For $s\in[0,a(x)]=[0,\frac{1}{4}],$
then
\begin{align*}
&\rho_x(\{y: f_{q}(x)\leq y\leq f_{q}(x)+s\})=\int_{f_q(x)}^{f_{q}(x)+s}\frac{d\rho_x(y)}{dy}dy
=2^{2\varphi+1}s^{\varphi+1}.
\end{align*}
By similarity, $\forall s\in[0,\frac{1}{4}],$
\begin{align*}
\rho_x(\{y: f_{q}(x)-s \leq y \leq f_{q}(x)\})=2^{2\varphi+1}s^{\varphi+1}.
\end{align*}
So we say that $\rho$ has a
$\infty$-average type $\varphi+1$.\\
Since $f_q\in {\cal H}_K$ and $K\in C^{\infty}(X\times X),$ then (\ref{app}) and (\ref{cover}) hold with $\be=1$ and $k=0.$
Thus, $\theta=\frac{2}{q+\varphi+1}$ and $r=q+\varphi+1$. Noting that the choice of $\la$ and $\epsilon$ satisfy (\ref{retrict})
and $\Lambda>0$.
This complements our Theorem 1.\\
{\it Proof of Corollary 1.} It is an easy consequence of Theorem 2.

\end{document}